\documentclass{amsart}

\usepackage{amssymb, amsmath}
\usepackage{mathrsfs}
\usepackage{amscd}
\usepackage{verbatim}
\usepackage{a4wide}

\usepackage[utf8]{inputenc}
\usepackage[T1]{fontenc}
\usepackage{soul}

\usepackage{enumerate}

\allowdisplaybreaks

\usepackage{hyperref}

\theoremstyle{plain}
\newtheorem{theorem}{Theorem}[section]
\theoremstyle{remark}
\newtheorem{remark}[theorem]{Remark}

\theoremstyle{plain}
\newtheorem{corollary}[theorem]{Corollary}
\newtheorem{lemma}[theorem]{Lemma}
\newtheorem{proposition}[theorem]{Proposition}
\newtheorem{definition}[theorem]{Definition}

\numberwithin{equation}{section}

\def\N{{\mathbb N}}

\def\R{{\mathbb R}}

\newcommand{\E}{{\mathbb E}}
\renewcommand{\P}{{\mathbb P}}
\newcommand{\F}{{\mathscr F}}
\newcommand{\A}{{\mathscr A}}

\renewcommand{\O}{\Omega}
\newcommand{\g}{\gamma}

\newcommand{\calL}{{\mathscr L}}

\newcommand{\one}{{{\bf 1}}}

\newcommand{\lb}{\langle}
\newcommand{\rb}{\rangle}

\newcommand{\limn}{\lim_{n\to\infty}}

\newcommand{\umd}{\textsc{umd} }

\begin{document}

\title[stochastic integrability in infinite dimensions]{Weak characterizations of stochastic integrability and Dudley's theorem in infinite dimensions}

\author{Martin Ondrej\'at}
\address{Institute of Information Theory and Automation\\ Academy of Sciences Czech Republic\\ Pod Vod\'arenskou v\v e\v z\'{\i} 4 \\ 182 08 Prague 8\\ Czech Republic} \email{ondrejat@utia.cas.cz}

\author{Mark Veraar}
\address{Delft Institute of Applied Mathematics\\
Delft University of Technology \\ P.O. Box 5031\\ 2600 GA Delft\\The
Netherlands} \email{M.C.Veraar@tudelft.nl}

\keywords{Stochastic integration in Banach spaces, UMD Banach spaces, cylindrical Brownian motion, random variables with polynomial tails, almost sure limit theorems, Dudley representation theorem, universal representation theorem, weak characterization of stochastic integrability, Doob representation theorem}

\subjclass[2000]{Primary: 60H05 Secondary: 28C20, 60B11, 60F99}

\thanks{The first named author is
supported by the GA\v CR grant P201/10/0752. The second author is supported by VENI subsidy 639.031.930
of the Netherlands Organisation for Scientific Research (NWO)}




\begin{abstract}
In this paper we consider stochastic integration with respect to cylindrical Brownian motion in infinite dimensional spaces. We study weak characterizations of stochastic integrability and present a natural continuation of results of van Neerven, Weis and the second named author. The limitation of weak characterizations will be demonstrated with a nontrivial counterexample. The second subject treated in the paper addresses representation theory for random variables in terms of stochastic integrals. In particular, we provide an infinite dimensional version of Dudley's  representation theorem for random variables and an extension of Doob's representation for martingales.
\end{abstract}

\maketitle

\section{Introduction}

Representations of martingales and random variables as stochastic integrals with respect to Wiener processes (or terminal values thereof) have been a classical issue of stochastic analysis. Representation results show the remarkable universality of the Brownian motion among all continuous local martingales and have a variety of applications in stochastic analysis including backward stochastic differential equations (e.g.\ \cite{HuPeng} and \cite{ParPeng}) and mathematical finance (e.g. in hedging, nonlinear pricing, continuous trading or optimal portfolio, see \cite{CoHu}, \cite{ElQ}, \cite{HuKen}, \cite{Pli86}).

To be more precise, let $(\Omega,\mathscr A,\P)$ be a probability space with filtration $\F = (\mathscr F_t)_{t\geq 0}$ and $W$ be a Brownian motion with respect to $\F$. One typically considers an $\mathscr F_\infty$-measurable random variable $\xi$ and asks whether there exists a jointly measurable adapted process $\phi$ with paths in $L^2(\R_+)$ such that $\int_0^\infty\phi(s)\,dW(s)=\xi$. Let us refer to this as to the ``terminal value representation problem'' hereafter. Often one also considers a continuous local $\F$-martingale $M$ with $M(0) =0$ and of a given quadratic variation of the form $\langle M\rangle_t=\int_0^t |\phi(s)|^2 \,ds$ and asks whether there exists a Brownian motion $W$ with respect to $\F$ such that $M_t=\int_0^t \phi(s)\,dW(s)$ for every $t\ge 0$. Let us refer to this as to the ``martingale representation problem''.
One may also consider a third kind of problems (which we will not pursue in this paper), when $M$ is a local $\F^W$-martingale with $M(0) = 0$ and the question is, whether there exists a stochastically integrable process $\phi$ such that $M_t=\int_0^t \phi(s)\,dW(s)$ for every $t\ge 0$. Let us refer to this as to the ``Brownian martingale representation problem'' and let us remark that this issue goes historically hand in hand with the terminal value representation problem.

As far as the terminal value representation problem is concerned, in the scalar case, when the terminal random variable is $\mathscr F^W_\infty$-measurable, consider the mapping $I:L^p_{\mathscr F^W}(\Omega;L^2(\R_+))\to L^p(\Omega,\mathscr F^W_\infty,\P)$ defined by
\[I \phi = \int_0^\infty \phi \, d W.\]
Then it is known that
\begin{enumerate}[(i)]
\item $I$ is an isomorphism onto the space of centered terminal states in $L^p(\Omega,\mathscr F^W_\infty,\P)$ for $p\in (1, \infty)$,
\item $I$ is a surjection when $0\le p<1$,
\item $I$ is an isomorphism onto the Hardy space $H^1$ in the sense of Garcia \cite{Garsia} when $p=1$.
\end{enumerate}
See Garling \cite{GarlingRange} for $p\in[0,\infty)\setminus\{1\}$ and \cite{BD63}, \cite{DG78}, \cite{Gil86}, \cite{Jacka88}, \cite{Ob04}, \cite{RY} for $p=1$. The case $p=2$ is due to It\^o \cite{Ito51} and Kunita-Watanabe \cite{KuWa} and $p=0$ to Dudley \cite{Dud77}.
The terminal value representation problem and the Brownian martingale representation problem have been recently studied and resolved for $E$-valued terminal conditions $\xi\in L^p(\O,\F^W_\infty;E)$  for $p\in (1, \infty)$ in \cite{NVW1} by van Neerven, Weis and the second named author under the assumption that $E$ is a \umd Banach space (see explanation below). As a consequence, in full generality of the \umd-setting it holds that if terminal random variables can be represented then local $\F^W$-martingales have continuous modifications.

The martingale representation problem was first solved by Doob \cite{Doob} in one dimension (extensions to finite dimensions are nowadays standard). Generalizations to continuous local martingales in Hilbert spaces have been given in \cite{LeOu} and \cite{Ou} by L{\'e}pingle and Ouvrard, to complete nuclear spaces in \cite{KoMart} by K{\"o}rezlio{\u{g}}lu and Martias and to Banach spaces in \cite{Dett90} by Dettweiler and in \cite{Ondr07} by the first named author. The martingale representation theorem has plenty of applications in the theory of stochastic equations, e.g.\ in the proof of coincidence of weak solutions of stochastic evolution equations and of solutions of the associated martingale problem of Stroock and Varadhan and in proofs of existence of weak solutions of stochastic evolution equations by compactness methods, see e.g.\ \cite{BG99}, \cite{CaG}, \cite{FlG} or \cite{GG94}.

In the setting of infinite dimensional spaces, i.e. when the terminal values or martingales take values in a Hilbert space or a Banach space, representation problems become more difficult and it turns out that they are closely connected with the problem of weak characterization of stochastic integrability (which do not appear when working in $\R^d$ case). Let us demonstrate this with the following example. Let $E$ be a Hilbert space, $W$ a standard Brownian motion, $p,q\in (0,\infty)$, $\phi$ a strongly $\F$-progressively measurable $E$-valued process such that
\begin{equation}\label{eq:introweak1}
\E\,\left(\int_0^1|\langle \phi(s), x\rangle|^2\,ds\right)^\frac p2<\infty,\qquad x\in E
\end{equation}
and $\xi\in L^q(\Omega;E)$ satisfies, for all $x\in E$, almost surely
\begin{equation}\label{eq:introweak2}
\int_0^1\langle \phi(s),x\rangle\,dW(s)=\langle \xi,x\rangle.
\end{equation}
Does this imply that $\phi$ is stochastically integrable in $E$ and, in particular, that $\int_0^1\phi(s)\,dW(s)=\xi$ almost surely? Surprisingly this is true for $p=q>1$ as proved in \cite{NVW1} in a more general setting which will be explained below. In particular, for $p=q>1$, \eqref{eq:introweak1} and \eqref{eq:introweak2} imply that $\phi\in L^p(\O;L^2(0,1;E))$ and then also $\int_0^1\phi(s)\,dW(s)=\xi$ (thus weak=strong). We will show that there is an extension to $p=q=1$ of this results. Moreover, we complete the picture by giving a counterexample in the case where $p<1$ and $q$ is arbitrary. Hence, the only cases which remain unanswered are $p\in [1, \infty)$ and $q\in (0,1]$.

In the already mentioned work \cite{NVW1}, a theory of stochastic integration for processes with values in a \umd Banach space $E$ has been developed using ideas of \cite{Ga1} by Garling and \cite{MC} by McConnell. The class of \umd Banach space have been extensively studied by Burkholder in the eighties (see \cite{Bu3} and references therein). The spaces $L^p$ and Sobolev spaces $W^{s,p}$ with $p\in (1, \infty)$ and $s\in \R$ are examples of spaces which have \textsc{umd}. As a rule of thumb one could say that all ``classical'' reflexive spaces have \textsc{umd}. The $L^p$ and $W^{s,p}$-spaces play a central role in the theory of stochastic evolution equations. In \cite{NVW10} it has been shown how the stochastic integration theory can be combined with functional calculus tools to obtain maximal regularity results for large classes of stochastic PDEs.

Solutions to stochastic evolution equations and in particular stochastic PDEs are often formulated in a weak sense (see \cite{DPZ} and \cite{Kry}). Therefore, weak characterizations of stochastic integrals can be helpful in obtaining stochastic integrability. The weak characterizations of \cite{NVW1} have already played an important role in proving vector-valued versions of It\^o's formula in \cite{BNVW}. Moreover, for a concrete SPDE, stochastic integrability has been checked using this technique in \cite{BrzVer11}. Finally, let us mention that weak identities for stochastic integrals appear in the setting of martingale solutions to stochastic evolution equations (see \cite{KunzeMart} and \cite{Ondrejat04}).

\medskip

\noindent \textbf{Overview: }
The aim of this paper is to study weak characterizations of stochastic integrability in respect of representation theorems in Banach spaces. We prove positive results on weak characterizations in Theorem \ref{thm:bddpaths} and Corollary \ref{cor:L1} and a negative result in Theorem \ref{thm:main}.
We also prove in Theorem \ref{thm:Dudley} a general terminal value representation theorem in every Banach space and for every filtration (i.e. not necessarily for the Brownian one) and we give an example of a universal indefinite stochastic integral whose set of $L^0$-accumulation points at infinity coincides with $L^0(\Omega;E)$, which is new even in the scalar setting. Finally, we prove a combination of a martingale representation theorem and a weak characterization of stochastic integrability in \umd Banach spaces in Theorem \ref{thm:Doob}. Background material on stochastic integration and path regularity of martingales in Banach spaces is collected in the Appendix.

\medskip

\noindent \textbf{Convention:} Below all vector spaces will be considered over the real scalar field. As complex vector spaces can be viewed as real vector spaces, this is not really a restriction.
If not stated otherwise $(\O, \A,\P)$ is a probability space with a complete filtration $\F = (\F_t)_{t\geq 0}$. Moreover, $W$ and $W_H$ will be a Brownian motion and a cylindrical Brownian motion with respect to $\F$ respectively. Here $H$ is a separable real Hilbert space.

Let $E$ be a Banach space. We write $L^0(\O;E)$ for the space of strongly measurable mappings $\xi:\O\to E$. Note that $L^0(\O;E)$ is a complete metric space with translation invariant metric $\|\xi\|_{L^0(\O;E)} = \E(\|\xi\|\wedge 1)$. Moreover, convergence in $L^0(\O;E)$ coincide with convergence in probability.

Let $\calL(H,E)$ denote the bounded linear operators from $H$ into $E$. Let $T\in (0,\infty]$. A process $\Phi:[0,T]\times\O\to \calL(H,E)$ is called {\em $H$-strongly measurable} if for all $h\in H$, $(t,\omega)\mapsto \Phi(t,\omega) h$ is strongly measurable. One defines the notion of $H$-strongly adaptedness and $H$-strongly progressive measurability in a similar way.
Let $p\in [0,\infty)$. An $H$-strongly measurable process $\Phi:[0,T]\times\O\to \calL(H,E)$
is said to be {\it weakly} $L^p(\O;L^2(0,T;H))$ if for all $x^*\in E^*$, the process
$(\Phi^* x^*)(t,\omega) := \Phi(t,\omega)^*x^*$ belongs to $L^p(\O;L^2(0,T;H))$. In the special case that $H = \R$, we identify $\calL(H,E)$ with $E$ and in that case we also write $\lb \Phi, x^*\rb$ instead of $\Phi^* x^*$.

\section{A counterexample to weak characterizations of stochastic integrability\label{sec:counter}}
In our first result we give a counterexample to weak characterizations of stochastic integrals in a Hilbert space. Counterexamples in more general Banach space will be discussed in Remark \ref{rem:Banachunconditional}. Positive results for $p\geq 1$, will be presented in Section \ref{sec:weak}.

\begin{remark}
A general definition of the stochastic integral can be found in Definition \ref{def:stochint}.
In the special case that $E$ is a Hilbert spaces and $H = \R$ and $\phi:[0,T]\times\O\to E$ is an adapted and strongly measurable process,
the definition reduces to the classical setting and is equivalent to $\phi\in L^2(0,T;E)$ almost surely. Indeed, this follows from the standard fact that for a sequence of adapted step processes $\Big(\int_0^\cdot \phi_n \, dW\Big)_{n\geq 1}$ converges in $L^0(\O;C_b([0,T];E))$ if and only if $(\phi_n)_{n\geq 1}$ converges in in $L^0(\O;L^2(0,T;E))$ (see \cite[Proposition 17.6]{Kal} for the scalar case).
\end{remark}

\begin{theorem}\label{thm:main}
Let $E$ be an infinite dimensional Hilbert space. There exists a
strongly progressive process $\phi:[0,1]\times\O\to E$
which is weakly in $L^p(\O;L^2(0,1))$ for all $p\in [0,1)$ and satisfies for all $x\in E$,
\begin{equation}\label{eq:weak0}
\int_0^1 \lb \phi(t), x\rb \, d W(t) = 0, \ \text{almost surely},
\end{equation}
but $\|\phi\|_{L^2(0,1;E)} = \infty$ almost surely. In particular, $\phi$ is not stochastically integrable.
\end{theorem}
Some details on stochastic integration theory of vector-valued processes can be found in Appendix \ref{sec:stochint} where one should take $H = \R$, and $W_H = W$.

\begin{remark}\label{rem:L1}
If $\phi$ is weakly $L^1(\O;L^2(0,1))$ and satisfies \eqref{eq:weak0}, then $\phi=0$. Indeed, it follows from \eqref{eq:weak0} and the martingale property of the stochastic integral that for all $x\in E$ and all $t\in [0,1]$, $\int_0^t \lb \phi(s), x\rb \, d W(s) = \E\Big(\int_0^1 \lb \phi(s), x\rb \, d W(s)\Big|\F_t\Big) =0$ almost surely. Therefore, the Burkholder-Davis-Gundy inequality (see \cite[Proposition 17.7]{Kal}) yields that for all $x\in E$, $\|\lb \phi, x\rb\|_{L^1(\O;L^2(0,1))} =0$ and hence $\phi = 0$ a.s.
\end{remark}

From \cite{Dud77}, one can deduce that there exists a progressively measurable process $\psi\in L^0(\O;L^2(0,1))$ which satisfies
(see also \cite[p.\ 196]{Steele})
\[\int_0^{1/2} \psi \, d W = W(1/2) \ \ \ \  \text{and} \ \ \ \int_0^1 \psi \, dW =0.\]
This process plays an important role in the construction in Theorem \ref{thm:main}. Before we turn to the proof we present two lemmas. In the first lemma, we calculate the distribution of $\xi=\|\psi\|_{L^2(0,1)}^2$ using arguments of \cite{Dud77}, where the tail of $\xi$ was estimated from above.

Let the function $f_b:[0,b)\to \R$ be given by $f_b(x) = (b-x)^{-1}$ if $0<b<\infty$ and $f_b(x) = 1$ if $b=\infty$.
\begin{lemma}\label{lem:distr}
Let $0\leq a<b\leq \infty$. Let $\eta:\O\to \R$ be $\F_a$-measurable. Let $Y:[a, b)\times\O\to \R$, the stopping time $\tau$ and $h:[a,b)\to \R$ be defined by
\[Y(r) = \int_{a}^r f_b(s) \, d W(s), \ \ \tau = \inf\Big\{r\geq a: Y(r) =
\eta \Big\}, \ \ \text{and} \ \  h(r) = \int_{a}^r f_b(s)^2  \, ds\]
Then $a\leq \tau<b$ a.s.\ and
\[\sqrt{2/(\pi e)} \,  \E \min\{|\eta|/\sqrt{t}, 1\}\leq \P(h(\tau)>t) \leq \E \min\{|\eta|/\sqrt{t}, 1\},  \ \ t>0.\]
\end{lemma}

\begin{proof}
One has that $\tau$ is a stopping time with $a\leq \tau<b$ a.s. (see \cite[p. 192]{Steele}).
As in \cite[p. 192]{Steele} one sees that $Y$ and $(W(h(t)))_{t\in
[a,b)}$ are identically distributed. Moreover, $\eta$ and $Y$ are independent. Let $\mu_{\eta}$ be the distribution function of $\eta$.
It follows that
\begin{align*}
\P(h(\tau)>t) &= \P(\tau>h^{-1}(t))
\\ & = \P\big(Y(s) \neq \eta \ \text{for all $s\in [a, h^{-1}(t)]$}\big)
\\ & = \P\big(Y(h^{-1}(s)) \neq \eta \ \text{for all $s\in [0,t]$}\big)
\\ & = \int_{-\infty}^\infty \P\big(Y(h^{-1}(s))  \neq x \ \text{for all $s\in [0,t]$}\big) \, d\mu_{\eta}(x)
\\ & = \int_{-\infty}^\infty \P\big(W(s) \neq x \ \text{for all $s\in [0,t]$}\big) \, d\mu_{\eta}(x)
= \int_{-\infty}^\infty \P(\tau_x>t) \, d\mu_{\eta}(x)
\end{align*}
where $\tau_x = \inf\{s>0: W(s) = x\}$ for $x\in \R$. The random variable
$\tau_x$ has density $f_x(s) = (2\pi)^{-1/2} s^{-3/2} |x|
\exp(-x^2/(2s))$ (see \cite[Section III.3]{RY}) for $x\ne 0$ and $\tau_0=0$ a.s.
It follows that
\begin{align*}
\P(h(\tau)>t) & = (2\pi)^{-1/2}\int_{-\infty}^\infty \int_t^\infty  s^{-3/2} |x| \exp(-x^2/(2s))  \, ds \, d\mu_{\eta}(x)
\\ & = 2^{1/2} \pi^{-1/2}\int_{-\infty}^\infty   \int_{0}^{|x|/\sqrt{t}}    \exp(-r^2/2) \, dr \, d\mu_{\eta}(x)
\\ & \leq   2^{1/2} \pi^{-1/2}\int_{-\infty}^\infty  \min\{|x|/\sqrt{t}, (\pi/2)^{1/2}\} d\mu_{\eta}(x)
\end{align*}
and
\[\P(h(\tau)>t)  \geq 2^{1/2} \pi^{-1/2} e^{-1/2} \int_{-\infty}^\infty  \min\{|x|/\sqrt{t}, 1\} d\mu_{\eta}(x).\]
Clearly, the result follows from these two estimates.
\end{proof}

\begin{remark}\label{rem:Gaussian}
It follows from the proof of Lemma \ref{lem:distr} that if $\eta$ is a centered Gaussian random variable with $\E\eta^2=\sigma^2$, then the assertion of the lemma can be strengthened to
\begin{align*}
\P(h(\tau)>t)  & = \pi^{-1}\sigma^{-1}\int_{0}^\infty \int_t^\infty  s^{-3/2} x \exp(-x^2/(2s))  \, ds \, e^{-x^2 \sigma^{-2}/2} \, dx
\\ &= \pi^{-1}\sigma^{-1}\int_t^\infty   s^{-3/2} \frac{1}{\frac 1{\sigma^2}+\frac1s} \, ds
= 2/\pi\int_0^{\sigma/\sqrt{t}}  \frac{1}{1+u^2} \, du = 2/\pi \arctan(\sigma/\sqrt{t}).
\end{align*}
\end{remark}

In the next lemma we study summability behavior of independent sequences of positive random variables with polynomial tails. In particular, the random variable $h(\tau)^{1/2}$ of Lemma \ref{lem:distr} often satisfies the conditions below.

\begin{lemma}\label{lem:charact}
Let $(\xi_n)_{n\geq 1}$ be a sequence of independent $[0,\infty)$-valued random variables for which there is a constant $C>0$ such that
\begin{align}\label{eq:probineqtail}
C^{-1} (t+1)^{-1}  \leq \P(\xi_1 >t) \leq C (t+1)^{-1}, \ \ \  \ t>0
\end{align}
For a sequence $(c_n)_{n\geq 1}$ in $(0,1]$ one has $\sup_{n\geq 1} c_n \xi_n<\infty$ a.s.\ if and only if $\sum_{n\geq 1} c_n<\infty$.
\end{lemma}

\begin{remark} \
\begin{enumerate}
\item The equivalence of Lemma \ref{lem:charact} holds for any sequence of positive numbers $(c_n)_{n\geq 1}$, but this requires an additional argument and we will not need this generality. One can also show that the assertions of Lemma \ref{lem:charact} are equivalent to: for all $p\in (1, \infty]$, $\|(c_n \xi_n)_{n\geq 1}\|_{\ell^p}<\infty$ a.s. Moreover, this result is sharp in the sense that for $p=1$ there are counterexamples.
\item For the case of equivalences of moments of random variables such as the above ones, we refer to \cite{GoLiScWe} for detailed discussions.
\end{enumerate}
\end{remark}

\begin{proof}[Proof of Lemma \ref{lem:charact}]
Let $p_n(\lambda) = \P(\xi_n>c_n^{-1} \lambda)$ for $\lambda>0$. Observe that for all $\lambda>0$ one has
\begin{equation}\label{eq:prodsup}
\P(\sup_{n\geq 1} c_n \xi_n\leq \lambda)
 =  \prod_{n\geq 1} (1 - p_n(\lambda)) =\exp\Big(\sum_{n\geq 1}  \log\big(1 - p_n(\lambda) \big)\Big).
\end{equation}

Now first assume $\sum_{n\geq 1} c_n<\infty$. Let $\lambda> 2C$. Then $0\leq p_n(\lambda)\leq  C (c_n^{-1} \lambda+1)^{-1}<1/2$ for all $n\geq 1$. Therefore, one has
\[-\log\big(1 - p_n(\lambda) \big)\leq 2p_n(\lambda) \leq 2C (\lambda c_n^{-1}+1)^{-1}\leq 2C \lambda^{-1} c_n.\]
It follows that
\[\sum_{n\geq 1} \log\big(1 - p_n(\lambda) \big) \geq -\frac{c}{2} \lambda^{-1} \sum_{n\geq 1} c_n = -K\lambda^{-1},\]
where $K$ is a constant not depending on $\lambda$. From \eqref{eq:prodsup} we obtain that
\[\P(\sup_{n\geq 1} c_n \xi_n\leq \lambda) \geq \exp(-K\lambda^{-1}).\]
Letting $\lambda\to \infty$, one sees that $\P(\sup_{n\geq 1} c_n \xi_n<\infty)=1$ and the result follows.

Next assume $\sup_{n\geq 1} c_n \xi_n<\infty$ almost surely. Choose $\lambda\geq 1$ such that $\P(\sup_{n\geq 1} c_n \xi_n<\lambda)\geq 1/2$. Then it follows that the series on the right-hand side of \eqref{eq:prodsup} converges and moreover
\[\log(1/2) \leq \sum_{n\geq 1}  \log\big(1 - p_n(\lambda) \big) \leq -\sum_{n\geq 1} p_n(\lambda) \leq - C^{-1} \sum_{n\geq 1}  (c_n^{-1} \lambda +1)^{-1} \leq -C^{-1} \lambda^{-1}\sum_{n\geq 1}  c_n.\]
Hence $\sum_{n\geq 1}  c_n\leq C\lambda \log(2)$.
\end{proof}

Next we turn to the proof of Theorem \ref{thm:main}. The process $\phi$ will be constructed in such a way that the interval $(0,1)$ will be decomposed to an infinite number of random subintervals, and on each of theses subintervals, $\phi$ will be a process so that the stochastic
integral over $(0,1)$ vanishes and, simultaneously, the $L^2(0,1;E)$-norm explodes.

\begin{proof}[Proof of Theorem \ref{thm:main}]
Let $(I_n)_{n\geq 1}$ be nonempty open disjoint intervals in $(0,1)$.
Write $I_n = (s_n, b_n)$ and let $a_n$ be the center of $I_n$ for each $n\geq 1$.

Now for each $n\geq1$, let $Y_n:[a_n, b_n)\times\O\to \R$, the $\F^{W}$-stopping time $\tau_n$ and $h_n:[a_n,b_n)\to \R$ be defined
as in Lemma \ref{lem:distr} with $a = a_n$, $b = b_n$ and $\eta = W(a_n) -W(s_n)$.
Then $a_n\leq \tau_n<b_n$ a.s.\ and by Remark \ref{rem:Gaussian} one has that
\begin{align*}
\P(\theta_n>t) = \P(h_n(\tau_n) >(a_n-s_n)t) = \frac{2}{\pi} \arctan(1/\sqrt{t}), \ \ \ t>0,
\end{align*}
where $\theta_n = h_n(\tau_n)/(a_n-s_n)$.
Moreover, $(\theta_n)_{n\geq 1}$ is sequence of independent random variables.
Let $\phi_n:\R\to \R$ be defined by
\[\phi_n = -\one_{[s_n, a_n)}(s) + \one_{[a_n,\tau_n]}(s)\frac{1}{b_n-s}.\]
Then $\phi_n$ is $\F^W$-progressively measurable, $\|\phi_n\|_{L^2(0,1)}^2 = (a_n-s_n)(1+\theta_n)$ and
\begin{align*}
\int_{0}^1 \phi_n \, dW& = \int_0^{1} \Big(-\one_{[s_n, a_n)}(s) + \one_{[a_n,\tau_n]}(s) (b_n-s)^{-1} \Big)\, dW(s)\\&  = -(W(a_n) - W(s_n)) + W(a_n)-W(s_n) = 0.
\end{align*}
Letting $\xi_n := (1+\theta_n)^{1/2}$, one sees that $(\xi_{n})_{n\geq 1}$ satisfies the conditions of Lemma \ref{lem:charact} with $\alpha = 1$. Let $c_n = n^{-1}$.
Let $(e_n)_{n\geq 0}$ be an orthonormal sequence in $E$. Let $\phi:[0,1]\to E$ be defined by
\[\phi = \sum_{n\geq 1} (a_n -s_n)^{-1/2} c_n \phi_n e_n.\]
Let $p\in (0,1)$. We prove that $\phi$ is weakly $L^p(\O;L^2(0,1))$. By Jensen's inequality it suffices to consider $2/3<p<1$. Since the $(\phi_n)_{n\geq 0}$ have disjoint supports, we find that for all $x\in E$, (letting $x_n = \lb x, e_n\rb$ for $n\geq 1$)
\begin{equation}\label{eq:weakL2}
\begin{aligned}
\|\lb \phi, x\rb\|_{L^2(0,1)} & = \Big(\int_{0}^1 \Big|\sum_{n\geq 0} (a_n -s_n)^{-1/2} x_n c_n \phi_n \Big|^2 \, dt\Big)^{1/2}
\\ & = \Big(\sum_{n\geq 0} (a_n -s_n)^{-1}  x_n^2 c_n^2 \int_{0}^1 \phi_n^2 \, dt\Big)^{1/2}
\\ &  = \Big(\sum_{n\geq 0} x_n^2 c_n^2 \xi_n^2\Big)^{1/2}
 \leq \Big(\sum_{n\geq 0} |x_n|^p |c_n|^p \xi_n^p\Big)^{1/p}
\end{aligned}
\end{equation}
where we used $\|y\|_{\ell^2}\leq \|y\|_{\ell^{p}}$. Note that $C_p = \E(\xi_n^p) = \E(\xi_1^p)$ is finite. Taking $p$-th moments on both sides of the above estimate it follows that
\begin{equation}\label{eq:weakLpL2}
\E\|\lb \phi, x\rb\|_{L^2(0,1)}^p \leq C_p \sum_{n\geq 0} |x_n|^p |c_n|^p \leq C_p \|x\|_{\ell^2}^p \|(c_n)_{n\geq 1}\|_{\ell^{pr}}^p<\infty,
\end{equation}
were we applied H\"older's inequality with $\frac{p}{2} + \frac{1}{r} =1$, and we used $pr = \frac{2p}{2-p}>1$. Moreover,
\begin{equation}\label{eq:weakequal0}
\int_{0}^1 \lb \phi, x\rb \, d W = \int_0^1 \Big(\sum_{n\geq 0} (a_n -s_n)^{-1/2} x_n c_n \phi_n\Big) \, d W = \sum_{n\geq 0}  (a_n -s_n)^{-1/2}  c_n x_n  \int_0^1 \phi_n\, d W =0.
\end{equation}
Finally,
\begin{equation}\label{eq:notL2}
\int_{0}^1 \|\phi\|^2_E \, dt = \int_0^1 \sum_{n\geq 0} (a_n -s_n)^{-1} c_n^2 |\phi_n|^2 \, dt = \sum_{n\geq 0} (a_n -s_n)^{-1} c_n^2 \int_0^1  |\phi_n|^2 \, dt = \sum_{n\geq 0} c_n^2 \xi_n^2
\end{equation}
and the latter is a.s.\ infinite by Lemma \ref{lem:charact} in view of the embedding $\ell^2\hookrightarrow\ell^\infty$.
\end{proof}

\begin{remark}\label{rem:Banachunconditional}
The construction in Theorem \ref{thm:main} can be extended to the setting where $E$ is a \umd Banach space with a normalized unconditional basic sequence $(e_n)_{n\geq 1}$.  The statement of Theorem \ref{thm:main} should then be replaced by:
there is a progressively measurable $\phi:[0,1]\times\O\to E$ which is weakly in $L^p(\O;L^2(0,1))$ for all $p\in [0,1)$ and for all $x^*\in E^*$
\[\int_0^1 \lb \phi(t), x^*\rb \, d W(t) = 0, \ \text{almost surely}.\]
and $\|\phi\|_{\g(L^2(0,1),E)} = \infty$, almost surely. In particular, $\phi$ is not stochastically integrable (in the sense of Appendix \ref{sec:stochint}).
Not every Banach space contains an unconditional basic sequence (see \cite{GoMa}). However, spaces such as $E = L^w(S)$, where $(S,\Sigma, \mu)$ is a measure space and $w\in (1, \infty)$ are {\textsc{umd}.} Moreover, if $E = L^w(S)$ is infinite dimensional, it contains a normalized unconditional basic sequence $e_n = \mu(S_n)^{-1/w} \one_{S_n}$, with $(S_n)_{n\geq 1}$ of finite measure and disjoint.

Let us briefly explain the proof of the above statement. For precise definitions on type of Banach spaces and unconditionality we refer to \cite{AlKa}. Since $E$ is a \umd  space it follows from the Maurey-Pisier theorem that $E$ has type $s$ with constant $T_s$ for some $s\in (1, 2]$. Let $(r_n)_{n\geq 1}$ and $(\g_n)_{n\geq 1}$ be a Rademacher and standard Gaussian sequence on some probability space $(\O',\F',\P')$ respectively. We claim that for all $x^*\in E^*$ one has $(x_n)_{n\geq 1}$ is in $\ell^{s^\prime}$, where $x_n = \lb e_n, x^*\rb$ and where $s'$ denotes the conjugate exponent of $s$. Indeed, let $a\in \ell^{s}$, Then
\begin{align*}
\sum_{n\geq 1} a_n x_n & \leq \|x^*\| \Big\|\sum_{n\geq 1} a_n e_n \Big\| \leq K\|x^*\| \Big\|\sum_{n\geq 1} r_n a_n e_n \Big\|_{L^2(\O';E)} \\ & \leq T_s K\|x^*\| \Big(\sum_{n\geq 1} \|a_n e_n\|^s\Big)^{1/s} = T_s K\|x^*\| \, \|(a_n)_{n\geq 1}\|_{\ell^s},
\end{align*}
where $K$ is the unconditionality constant of $(e_n)_{n\geq 1}$. Taking the supremum over all $\|(a_n)_{n\geq 1}\|_{\ell^s}\leq 1$, the claim follows.

With the same notations as in Theorem \ref{thm:main}, \eqref{eq:weakL2} and \eqref{eq:weakequal0} with $\lb \phi, x\rb$ replaced by $\lb \phi, x^*\rb$, remain valid. Indeed, by Jensen's inequality it suffices to consider $\frac 1{2-s^{-1}}<p<1$. By the above claim one can replace \eqref{eq:weakLpL2} with
\begin{align*}
\E\|\lb \phi, x^*\rb\|_{L^2(0,1)}^p \leq C_p \sum_{n\geq 0} |x_n|^p |c_n|^p \leq C_p \|x\|_{\ell^{s'}}^p \|(c_n)_{n\geq 1}\|_{\ell^{pr}}^p<\infty,
\end{align*}
where we applied H\"older's inequality with $\frac{p}{s'} + \frac1r =1$.

Finally, using Proposition \ref{prop:twosided} and the definition of the $\g$-norm, \eqref{eq:notL2} can be replaced by
\[\|\phi\|_{\g(L^2(0,1),E)} = \Big\|\sum_{n\geq 0} \gamma_n  c_n \xi_n  e_n \Big\|_{L^2(\O';E)} \geq \sup_{n\geq 1} c_n \xi_n\]
and the latter is a.s.\ infinite by Lemma \ref{lem:charact}.
\end{remark}

\section{Dudley's representation theorem\label{sec:represenDudley}}

In \cite{Dud77}, Dudley obtained a representation theorem for real-valued random variables $\xi\in L^0(\O)$ which are $\sigma(W(t): t\leq 1)$-measurable where $W$ is a standard Brownian motion. In \cite{ESY}, this was further extended to other classes of martingales for more general filtrations which are not necessarily Brownian.
Obviously, Dudley's result extends to the case of $\xi\in L^0(\O;\R^d)$ using coordinatewise representations. It seems that an infinite dimensional version is not available yet and the original proof does not extend to this setting. Moreover, if $E=\ell^2$ a coordinatewise representation leads to divergent series in general. Below we provide a representation theorem in arbitrary Banach spaces. A brief survey of stochastic integration theory in Banach spaces and the definition of $\g(L^2(a,b),E)$ can be found in Appendix \ref{sec:stochint}.

\begin{theorem}\label{thm:Dudley}
Let $E$ be a Banach space and $0\leq a<b\leq \infty$. Let $\xi:\O\to E$ be strongly measurable. Then there exists a strongly progressively measurable process $\phi:[a,b]\times\O\to E$ which is stochastically integrable and satisfies
\[\int_a^b \phi \, d W = \xi\]
if and only if $\xi$ is $\sigma(\F_t: t<b)$-measurable.
\end{theorem}

\begin{remark}\
\begin{enumerate}
\item If $\F$ is left continuous or $b=\infty$ (e.g. Brownian case), then one has $\F_b=\sigma(\F_t: t<b)$.
\item It will be clear from the proof that $\phi\in L^0(\O;\g(L^2(a,b), E))\cap L^0(\O;L^2(a,b;E))$. Moreover, because of the special structure of the process, no geometric conditions on $E$ are required for its stochastic integrability.
\item Note that, as in the scalar case, the representation is not unique (see the text below Theorem \ref{thm:main}). Moreover, in general $\Big(\int_a^t \phi \, d W\Big)_{t\in[a,b]}$ will not be a martingale, but only a continuous local martingale.
\item If $\phi_1,\phi_2$ are $E$-valued, strongly measurable and adapted, weakly $L^1(\O;L^2(a,b))$ and satisfy
\[\int_a^b \phi_i \, d W = \xi,  \ \ \ i=1, 2,\]
then $\phi_1=\phi_2$ a.e. Indeed, this follows as in Remark \ref{rem:L1}.
\end{enumerate}
\end{remark}

\begin{proof}

{\em Step 1:} Let $Y$, $\tau$ and $h$ be defined as in Lemma \ref{lem:distr} with $\eta=1$. Then
\[C^{-1}\min\{t^{-1/2}, 1\}\leq \P(h(\tau)>t) \leq C \min\{t^{-1/2}, 1\},  \ \ t>0.\]
In particular, for any random variable $\theta:\O\to \R_+$ which is independent of $h(\tau)$, one has
\begin{equation}\label{eq:thetat}
C^{-1}\E \min\{\theta^{1/2} t^{-1/2}, 1\}\leq \P(\theta h(\tau)>t) \leq C\E \min\{\theta^{1/2} t^{-1/2}, 1\},  \ \ t>0.
\end{equation}
Note that
\begin{align*}
\int_0^1 \E \min\{\theta^{1/2} t^{-1/2}, 1\} \, dt &= \|\theta\|_{L^0(\O)} + \E \int_{\theta\wedge 1}^1 \theta^{1/2} t^{-1/2} \, dt
\\ & \leq \|\theta\|_{L^0(\O)} +  2  \E(\theta^{1/2}\wedge 1)  \leq 3 \|\theta\|_{L^0(\O)}^{1/2}
\end{align*}
where we used Jensen's inequality and $x\leq x^{1/2}$ for $x\in [0,1]$. It follows that (with a different $C$)
\begin{equation}\label{eq:L0afschatting}
C^{-1} \|\theta\|_{L^0(\O)}\leq \|\theta h(\tau)\|_{L^0(\O)} \leq C \|\theta\|_{L^0(\O)}^{1/2}.
\end{equation}

{\em Step 2:} Let $[a,b] := [a, \infty)$ if $b=\infty$. Now let $\xi:\O\to E$ be as in the theorem. Choose a sequence $(a_n)_{n\geq 1}$ in $[a,b]$ with $a_n\uparrow b$ and random variables $(\xi_n)_{n\geq 1}$ such that for each $n\geq 1$, $\xi_n$ is $\F_{a_n}$-measurable and $\|\xi - \xi_n\|_{L^0(\O;E)}<4^{-n}$. Let $\xi_0=0$.
For each $n\geq 1$, let $Y_n$, $\tau_n$ and $h_n$ be defined as in step 1 with $a=a_n$ and $b=a_{n+1}$. For each $n\geq 1$, define $\phi_n:[a,b]\to E$ by $\phi_n(t) = (\xi_{n} - \xi_{n-1})\one_{(a_{n}, \tau_{n})}(t) (a_{n+1}-t)^{-1}$. Then each $\phi_n$ is stochastically integrable and $\zeta_n(b)= \xi_n-\xi_{n-1}$, where $\zeta_n = \int_a^\cdot \phi_n \, d W$.
Note that
\[\eta_n:=\|(\xi_{n} - \xi_{n-1})\|_E h_n(\tau_n)^{1/2} = \|\phi_n\|_{\g(L^2(a,b), E)} = \|\phi_n\|_{L^2(a,b;E)}.\]
Since $\xi_n-\xi_{n-1}$ and $h_n(\tau_n)$ are independent it follows from \eqref{eq:L0afschatting} that
\begin{align*}
\sum_{n\geq 1} \|\eta_n\|_{L^0(\O)} & \leq C \sum_{n\geq 1} \|\xi_n - \xi_{n-1}\|_{L^0(\O)}^{1/2}
\\ & \leq C\|\xi_1\|_{L^0(\O;E)} + C \sum_{n\geq 2} (4^{-n} + 4^{-n+1})^{1/2} = C\|\xi_1\|_{L^0(\O;E)} + C 5^{1/2} 2^{-1}.
\end{align*}
By completeness it follows that $\sum_{n\geq 1} \eta_n$ converges in $L^0(\O)$. Therefore, the series $\phi := \sum_{n\geq 1} \phi_n$ converges (absolutely) in $\g(L^2(a,b), E)$ and in $L^2(a,b;E)$ almost surely. Since each $\phi_n$ is strongly progressively measurable it follows that $\phi$ is strongly progressively measurable.

Note that by the special form of $\phi_n$, for all $n\geq1$ one has
\[\|\zeta_n(t)\| = \Big|\int_0^t \|\phi_n\| \, d W\Big|,  \ \ \ t\in [a,b].\]
Let $\varepsilon\in (0,1)$ be arbitrary. Let $\delta_n = (2/3)^n \varepsilon$. Then by Proposition \ref{prop:L0cont} with $p=2$ and $\beta_{2,\R} = 1$ and $c_2=1$ we obtain
\begin{align*}
\P\Big(\|\zeta_n\|_{C_b([a,b];E)}>\varepsilon\Big) =  \P\Big(\sup_{t\in [a,b]}\Big|\int_a^t \|\phi_n\| \, d W\Big|>\varepsilon\Big) \leq (4/9)^n + \P((3/2)^n \eta_n >\varepsilon)
\end{align*}
For $n\geq2$, integrating over $\varepsilon\in (0,1)$ yields
\begin{align*}
\|\zeta_n\|_{L^0(\O;C_b([a,b];E))} & \leq (4/9)^{n} + \|(3/2)^n \eta_n\|_{L^0(\O)}\leq (4/9)^{n} + (3/2)^n \|\eta_n\|_{L^0(\O)}
\\ & \leq (4/9)^{n} + (3/2)^n C \|\xi_n- \xi_{n-1}\|_{L^0(\O;E)}^{1/2} \leq (4/9)^{n} + C (3/4)^n.
\end{align*}
Hence $\sum_{n\geq 2}\|\zeta_n\|_{L^0(\O;C_b([a,b];E))}<\infty$ and therefore $\sum_{n\geq 1} \zeta_n$ converges in $L^0(\O;C_b([a,b];E))$.

Since each $\phi_n$ is stochastically integrable and both
\begin{enumerate}[(a)]
\item $\sum_{n=1}^\infty \phi_n = \phi$ in $L^0(\O;L^2(a,b;E))$,
\item $\int_0^\cdot \sum_{n=1}^N \phi_n \, d W = \sum_{n=1}^N \zeta_n$ converges in $L^0(\O;C_b([a,b];E))$ as $N\to \infty$,
\end{enumerate}
it follows that $\phi$ is stochastically integrable, and moreover,
\[\int_a^b \phi\, d W = \sum_{n\geq 1} \zeta_n(b)  = \sum_{n\geq 1} \xi_n - \xi_{n-1} = \xi.\]
\end{proof}

\begin{remark}\label{rem:Dudley}
If additionally to the assumptions of Theorem \ref{thm:Dudley} one knows that $\xi\in L^p(\O;E)$ with $p\in (0,1)$. Then the process $\phi$ can be chosen in $L^p(\O;\g(L^2(a,b), E))\cap L^p(\O;L^2(a,b;E))$. Indeed, choose $(\xi_n)_{n\geq 1}$ such that $\xi_n$ is $\F_{a_n}$-measurable and $\|\xi - \xi_n\|_{L^p(\O;E)}^p<4^{-n}$.
From this one can deduce that $\sum_{n\geq 1} \phi_n$ converges $\phi$ in $L^p(\O;\g(L^2(a,b), E))\cap L^p(\O;L^2(a,b;E))$ and $\lambda\otimes \P$-a.s. Moreover,
\[\int_a^\cdot \phi \, d W = \sum_{n\geq 1} \int_a^\cdot \phi_n \, d W = \zeta \ \ \text{with convergence in $L^p(\O;C([a,b];E))$}.\]
This yields an infinite dimensional version of Garling's result \cite{GarlingRange}.
To prove the above facts it suffices to note that from \eqref{eq:thetat} with $\theta$ replaced by $\theta^2$ one obtains
\begin{equation}\label{eq:thetat2}
C^{-1}\E \min\{\theta t^{-1/p}, 1\}\leq \P(\theta h(\tau)^{1/2}>t^{1/p}) \leq C\E \min\{\theta t^{-1/p}, 1\},  \ \ t>0.
\end{equation}
It is easy to check that
\[\int_0^\infty \E \min\{\theta t^{-1/p}, 1\} \, dt = \frac{1}{1-p} \|\theta\|_{L^p(\O)}^p.\]
Therefore, integrating \eqref{eq:thetat2} shows that $C^{-1}\|\theta\|_{L^p(\O)} \leq \|\theta h(\tau)^{1/2}\|_{L^p(\O)}\leq C\|\theta\|_{L^p(\O)}$ for a different constant $C$. Now one can proceed as in the proof of Theorem \ref{thm:Dudley} with $L^p(\O)$ instead of $L^0(\O)$ with translation invariant metric $\|\cdot\|_{L^p(\O)}^p$ and using Proposition \ref{prop:twosided} instead of Proposition \ref{prop:L0cont} for $E=\R$ and $p\in (0,1)$.
\end{remark}

The next result is new even in the scalar setting. It can be seen as a universal version of Dudley's representation theorem. Namely, we construct a locally stochastically integrable process and consider its indefinite stochastic integral as a continuous curve in $L^0$. Then the set of accumulation points of this curve in $L^0$ coincides with $L^0$.
\begin{theorem}\label{thm:Dudleyuniversal}
Let $E$ be a separable Banach space. Assume $\F_{\infty}$ is $\P$-countably generated. Then there exists a strongly predictable process $\phi:\R_+\times\O\to E$ which is locally stochastically integrable and such that the following holds
\begin{itemize}
\item for every $\xi\in L^0(\O,\F_{\infty};E)$ there exists an increasing sequence $(n_k)_{k\geq 0}$ of natural numbers such that
\[\lim_{k\to \infty} \zeta(n_k) = \xi \ \ \ \ \text{in $L^0(\O;E)$},\]
where $\zeta(t) = \int_0^t \phi \, d W$.
\end{itemize}
\end{theorem}
Clearly, $\phi$ is only {\em locally} stochastically integrable, because $\lim_{t\to \infty} \zeta(t)$ does not exist.

\begin{proof}
{\em  Step 1:} Let $\xi\in L^0(\O,\F_{\infty};E)$. We claim that there exist an $\F$-predictable process $\eta$ such that $\lim_{t\to \infty} \eta(t) = \xi$ almost surely.
By Proposition \ref{prop:Doobreg}, there exists a left-continuous $\mathscr F$-adapted process $\theta^\prime:\R_+\times\O\to E$ satisfying $\theta^\prime(t) = \E\big(\frac{\xi}{1+\|\xi\|}\big|\F_{t-}\big)$ a.s. for every $t>0$ and $\lim_{t\to \infty} \theta'(t)  = \frac{\xi}{1+\|\xi\|}$ almost surely.
Define $\eta:\R_+\times\O\to E$ by
\begin{equation}\label{eq:eta}
\eta(t) = \one_{\|\theta'(t)\|<1} \frac{\theta'(t)}{1-\|\theta'(t)\|}.
\end{equation}
Then $\eta$ is also predictable, and by the previous observation $\lim_{t\to \infty} \eta(t) = \xi$ almost surely.

{\em Step 2:} Since $\F_{\infty}$ is $\P$-countably generated and $E$ is separable, the space $L^0(\O,\F_{\infty};E)$ is separable as well. Let $(\xi_n)_{n\geq 0}$ be a dense sequence in $L^0(\O,\F_{\infty};E)$. For each $n\geq 0$, step 1 implies that there exists a predictable $\eta_n$ such that $\lim_{t\to \infty} \eta_n(t) = \xi_n$ almost surely.

Let $\alpha:\N\to \N$ be such that for each $j\geq 0$,
\begin{equation}\label{eq:alphafunctie}
\#\{k\in \N:\alpha(k) = j\}=\infty.
\end{equation} Let $\rho_n = \eta_{\alpha(n)}(n)$ for $n\geq 0$. Since $\eta$ is predictable, for each $n\geq 0$, $\rho_n$ is strongly $\F^{n}$-measurable. Moreover, the set of accumulation points of $(\rho_n)_{n\geq 0}$ is equal to $L^0(\O,\F_{\infty};E)$. Indeed, it suffices to show that each $\xi_j$ is an accumulation point of $(\rho_n)_{n\geq 0}$. This is obvious from \eqref{eq:alphafunctie} and step 1.

{\em Step 3:} By Theorem \ref{thm:Dudley}, for each $n\geq 2$ we can find a stochastic integrable process $\phi_n:[n, n+1]\to E$ such that $\int_n^{n+1} \phi_n \, d W =  \rho_{n-1} - \rho_{n-2}$. Define
$\phi:\R_+\to E$ by
\[\phi = \sum_{n\geq 2} \one_{[n, n+1)} \phi_n.\]
Then $\phi$ is predictable and locally stochastically integrable. Let $\zeta$ be its stochastic integral process. For each integer $k\geq 2$ one has
\begin{align*}
\zeta(k) & = \sum_{n=2}^k \int_{n}^{n+1} \phi_n \, d W = \sum_{n=2}^k  (\rho_n - \rho_{n-1})
 = \sum_{n=1}^k  (\rho_n - \rho_{n-1})  = \rho_k.
\end{align*}
Now the theorem follows from Step 2.
\end{proof}

\begin{remark}
Using the result of Remark \ref{rem:Dudley} one can modify the proof of Theorem \ref{thm:Dudleyuniversal} to obtain a universal representation theorem for $L^p(\O;E)$ with $p\in (0,1)$ in the sense that for every $\xi\in L^p(\O;E)$ one can find a sequence of natural numbers $(n_k)_{k\geq 1}$ such that $\lim_{k\to \infty} \zeta(n_k) = \xi$ in $L^p(\O;E)$. To proof this fact one can follow the arguments of the above proof with $\theta(t) = \E\big(\frac{\xi}{\|\xi\|^{1-p}}\big|\F_t\big)$ (convention $\frac{0}{0}=0$) and \eqref{eq:eta} replaced by $\eta(t) = \theta'(t) \|\theta'(t)\|^{\frac{1-p}{p}}$.
\end{remark}

\section{Weak characterizations of stochastic integrability\label{sec:weak}}
In this section we obtain two new weak characterizations of stochastic integrability for processes $\Phi:[0,T]\times\O\to \calL(H,E)$, where $E$ is a \umd Banach space and $H$ is a separable Hilbert space. For details on stochastic integration theory of vector-valued processes we refer to Appendix \ref{sec:stochint}.

The following strengthening of \cite[Theorem 5.9]{NVW1} holds, where the case with $\zeta\in L^0(\O;C_b([0,T];E))$ was considered.
\begin{theorem}\label{thm:bddpaths}
Let $E$ be a \umd Banach space. Let $\Phi:[0,T]\times\O\to \calL(H,E)$ be an $H$-strongly measurable and adapted process which is weakly in $L^0(\O;L^2(0,T;H))$. Let $\zeta:[0,T]\times\O \to E$ be a process with almost surely bounded paths. If for all $x^*\in E^*$ almost surely, one has
\[\int_0^t \Phi^* x^* \, d W_H = \lb \zeta(t), x^*\rb,  \ \ \ t\in [0,T],\]
then $\Phi$ represents an element in $L^0_{\F}(\O;\g(L^2(0,T;H),E))$, and almost surely one has
\begin{equation}\label{eq:zetastrong}
\int_0^t \Phi \, d W_H = \zeta(t),  \ \ \ t\in [0,T].
\end{equation}
Moreover, $\zeta$ is a local martingale with a.s.\ continuous paths.
\end{theorem}

Before we turn to the proof, we mention the following consequence of the above result. It extends a part of \cite[Theorem 3.6]{NVW1} to the case $p=1$.
\begin{corollary}\label{cor:L1}
Let $E$ be a \umd Banach space. Let $\Phi:[0,T]\times\O\to \calL(H,E)$ be an $H$-strongly measurable and adapted process which is weakly in $L^1(\O;L^2(0,T;H))$. Let $\xi\in L^1(\O;E)$. If for all $x^*\in E^*$ one has almost surely
\[\int_0^T \Phi^* x^* \, d W_H = \lb \xi, x^*\rb,\]
then $\Phi$ represents an element in $L^0_{\F}(\O;\g(L^2(0,T;H),E))$, and
\[\int_0^T \Phi \, d W_H = \xi.\]
Moreover, for all $q\in (0,1)$ one has $\Phi\in L^q(\O;\g(L^2(0,T;H),E))$.
\end{corollary}

\begin{remark}\
\begin{enumerate}
\item Note that Corollary \ref{cor:L1} is false if one only assumes that $\Phi$ is weakly in $L^p(\O;L^2(0,T;H))$ for all $p\in [0,1)$ and $\xi\in L^q(\O;E)$ for some $q\in [0,\infty]$ (see Theorem \ref{thm:main}).
\item We also do not know whether there is a positive result like Corollary \ref{cor:L1} if $\Phi$ is weakly in $L^1(\O;L^2(0,T))$ and $\xi\in L^0(\O;E)$.
\item In general, $\Phi \notin L^1(\O;\g(L^2(0,T;H),E))$ in Corollary \ref{cor:L1}. Indeed, even in scalar case this fails, because Doob's maximal $L^p$-inequality does not hold for $p=1$.
\end{enumerate}
\end{remark}

\begin{proof}[Proof of Theorem \ref{thm:bddpaths}]
The proof is an extension to $L^0$ of the ideas in \cite[Theorem 3.6]{NVW1} for the $L^p$-case.
As in \cite[Theorem 3.6]{NVW1} we can assume $E$ and $E^*$ are separable, and we can choose a dense sequence of functionals $(x_n^*)_{n\geq 1}$ in $B_X$. For each $n\geq 1$ let $F_n = \bigcap_{i=1}^n \text{ker}(x_i^*)$, $E_n = E/F_n$ and let $Q_n:E\to E_n$ be the quotient map. Note that each $E_n$ is finite dimensional. For each $n\geq 1$, let $\Phi_n:[0,T]\times\O\to \calL(H,E_n)$ be defined by $\Phi_n = Q_n \Phi$. Fix $n\geq 1$. Clearly $\Phi_n$ is weakly in $L^0(\O;L^2(0,T;H))$, $\Phi_n$ is stochastically integrable and one can check that
\[\int_0^\cdot \Phi_n \, d W_H = \zeta_n\qquad \textrm{almost surely},\]
where $\zeta_n:[0,T]\times\O\to E_n$ be defined by $\zeta_n = Q_n \zeta$. Since $E_n$ is finite dimensional $\Phi_n$ represents an element $R_{n}\in L^0_{\F}(\O;\g(L^2(0,T;H), E))$.
By Proposition \ref{prop:L0cont} for each $\varepsilon>0$ and $\delta>0$ one has
\begin{equation}\label{eq:tailest}
\begin{aligned}
\P(\xi_n >\varepsilon) & \leq \frac{c_2^2 \beta_{2,E}^2 \delta^2}{\varepsilon^2} + \P\Bigl(\sup_{t\in[0,T]}\|\zeta_n(t)\|\geq
\delta\Bigr) \leq \frac{c_2^2 \beta_{2,E}^2 \delta^2}{\varepsilon^2} + \P(\eta\geq
\delta),
\end{aligned}
\end{equation}
where we used $\beta_{p,E_n} \leq \beta_{p,E}$, $\|Q_n\|\leq 1$ and
\[\xi_n := \|R_n(\omega)\|_{\g(L^2(0,T;H),E_n)}, \ \ \ \ \eta=\sup_{t\in[0,T]}\|\zeta(t)\|.\]
We claim that for almost all $\omega\in \O$, $C(\omega) := \sup_{n\geq 1} \xi_n<\infty$.
To prove the claim let $\theta>0$. We will show that there exists an $\varepsilon>0$ such that $\P(C>\varepsilon)<\theta$.
As in \cite[Theorem 3.6]{NVW1} for all $1\leq m\leq n$ one has $\xi_m = \|R_m\|_{\g(L^2(0,T;H),E_m)}\leq \|R_n\|_{\g(L^2(0,T;H),E_n)} = \xi_n$ pointwise in $\O$.
Since $\zeta$ has bounded paths a.s., we can find $\delta>0$ so large that $\P(\eta\geq \delta)<\theta/2$. Choose
$\varepsilon>0$ be so large that $\frac{c_2^2\beta_{2,E}^2 \delta^2}{\varepsilon^2}<\theta/2$. Then it follows from \eqref{eq:tailest} and Fatou's lemma that \begin{align*}
\P(\sup_{n\geq 1}\xi_n >\varepsilon) & = \P(\lim_{n\to \infty}\xi_n >\varepsilon) \leq \liminf_{n\to \infty} \P(\xi_n >\varepsilon) \leq \theta
\end{align*}
and the claim follows. Now the proof of the fact that $\Phi$ represents an $R_{\Phi}\in L^0_{\F}(\O;\g(L^2(0,T;H),E))$ can be finished as in \cite[Theorem 3.6, step 3-5]{NVW1}. It follows from Proposition \ref{prop:L0cont} that $\Phi$ is stochastically integrable. The identity \eqref{eq:zetastrong} can be deduced by using the weak identity for $(x_n^*)_{n\geq 1}$.
\end{proof}

\begin{proof}[Proof of Corollary \ref{cor:L1}]
Let $\widetilde\zeta:[0,T]\times\O\to E$ be an $\mathscr F$-adapted process with left-continuous paths admitting right limits satisfying $\widetilde\zeta(t) = \E(\xi|\F_{t-})$ a.s. for every $t\ge 0$ (see Proposition \ref{prop:Doobreg}). In particular, $\widetilde{\zeta}$ has bounded paths almost surely. Furthermore, from the path continuity and martingale property of the stochastic integral one deduces that for all $x^*\in E^*$ and all $t\in [0,T]$, almost surely
\[\lb \widetilde{\zeta}(t), x^*\rb = \int_0^t \Phi^* x^* \, d W_H.\]
Note that the right-hand side being a stochastic integral process has a continuous version. Therefore,
by the left-continuity of the left-hand side we obtain that for all $x^*\in E^*$, almost surely
\[\lb \widetilde{\zeta}(t), x^*\rb  =  \int_0^t \Phi^* x^* \, d W_H, \ \ \ t\in [0,T].\]
Now the assertion that $\Phi\in L^0(\O;\g(L^2(0,T;H), E))$ follows from Theorem \ref{thm:bddpaths}. Moreover, the strong identity
\eqref{eq:zetastrong} implies
\[\int_0^t \Phi \, d W_H = \widetilde{\zeta}(t),  \ \ \ t\in [0,T].\]
and it follows that $\widetilde{\zeta}$ is an $\F$-martingale which is a.s.\ continuous.

By \cite[Proposition 7.15]{Kal} one has
\begin{align*}
\lambda \P(\sup_{t\in [0,T]} \|\widetilde{\zeta}(t)\|>\lambda) \leq \|\xi\|_{L^1(\O;E)},  \ \ \lambda>0.
\end{align*}
Hence for all $s>0$ one has
\begin{align*}
\E  \sup_{t\in [0,T]} \|\widetilde{\zeta}(t)\|^q &= \int_0^\infty \P(\sup_{t\in [0,T]} \|\widetilde{\zeta}(t)\|^q>\lambda) \, d\lambda
\\ & \leq s + \|\xi\|_{L^1(\O;E)} \int_s^\infty \lambda^{-1/q} \, d\lambda = s + \frac{q\|\xi\|_{L^1(\O;E)} }{1-q} s^{(q-1)/q}.
\end{align*}
Minimizing the latter over $s>0$ one obtains $s=\|\xi\|_{L^1(\O;E)}^{q}$, and therefore
\[\E  \sup_{t\in [0,T]} \|\widetilde{\zeta}(t)\|^q \leq (1-q)^{-1} \|\xi\|_{L^1(\O;E)}^{q}.\]
Now the last assertion follows from Proposition \ref{prop:twosided}.
\end{proof}

\section{Doob's representation theorem \label{sec:represenDoob}}
The following version of Doob's representation theorem holds for \umd Banach spaces:
\begin{theorem}\label{thm:Doob}
Let $H$ be a real separable Hilbert space and let $E$ be a \umd Banach space. Let $M:\R_+\times\O\to E$ be an a.s.\ bounded process such that for all $x^*\in E^*$, $\lb M, x^*\rb$ is a local martingale with a.s.\ continuous paths. If there exists a weakly $L^0(\O;L^2(\R_+;H))$-process $g:\R_+\times \O\to \calL(H,E)$ which is $H$-strongly measurable and adapted and such that for all $x^*\in E^*$, a.s. one has
\[[\lb M, x^*\rb]_t   =  \int_0^t \| g(s)^* x^*\|^2_H \, d s, \ \ \ \ t\in \R_+,\]
then $g\in L^0(\O;\g(L^2(\R_+;H), E))$, $M$ is a local martingale with a.s.\ continuous paths and there exists a cylindrical Brownian motion $W_H$ on an extended probability space such that a.s.
\begin{equation}\label{eq:represE}
M_t = \int_0^t g(s) \, d W_H(s), \ \ \ \ t\in \R_+.
\end{equation}
\end{theorem}
\begin{proof}
Without loss of generality we can assume $E$ is separable. Since $E$ is reflexive, $E^*$ is separable as well and we can choose a dense sequence $(x_m^*)_{m\geq 1}$ in the unit sphere $S_{E^*}$.

Define a sequence $(\mu_n)_{n\geq 1}$ of $[0,\infty]$-valued random variables by
\[\mu_n = \inf\{t\geq 0: \|M_t\|\geq n\}, \ \ \ n\geq 1.\]
For each $m\geq 1$ let the stopping times $(\tau_n^m)_{n\geq 1}$ be defined as
\[\tau_n^{m} = \inf\{t\geq 0: |\lb M_t, x^*_m\rb|\geq n\}, \ \ \ n\geq 1\]
Then $\tau_n^m\geq \mu_n$ everywhere. Let $\tau_n = \inf_{m\geq 1}\tau_n^m$. Then each $\tau_n$ is measurable and satisfies $\tau_n\geq \mu_n$. Therefore, $\lim_{n\to \infty} \tau_n = \infty$ a.s. Note that $\tau_n$ is possibly not a stopping time.

We claim that for all $n\geq 1$ and all $x^*\in E^*$ one has
\[\E \int_0^{\tau_n} \lb g(s), x^*\rb^2 \, d s \leq n^2 \|x^*\|^2.\]
By homogeneity, density and Fatou's lemma it suffices to consider $x^* = x_m^*$. Fix $m,n\geq 1$.
Then by \cite[Theorems 17.5 and 17.11]{Kal} almost surely one has
\[\E \int_0^{\tau_n} \lb g(s), x^*_m\rb^2 \, d s \leq \E \int_0^{\tau_n^{m}} \lb g(s), x^*_m\rb^2 \, d s = \E [\lb M, x^*_m\rb^{\tau_n^{m}}]_\infty = \E|\lb M, x^*_m\rb_{\tau_n^{m}}|^2 \leq n^2 \|x^*_m\|^2.\]
which proves the claim.

Now we can apply \cite[Theorem 3.1]{Ondr07} to obtain a cylindrical Brownian motion $W_H$ such that for all $x^*\in E^*$ a.s.
\[\lb M_t, x^*\rb = \int_0^t \lb g(s), x^*\rb  \, d W_H(s),  \ \ \ t\in [0,T].\]
Now it follows from Theorem \ref{thm:bddpaths} that $g\in L^0(\O;\g(L^2(\R_+;H), E))$ and moreover \eqref{eq:represE} holds. Also, $M$ is a continuous local martingale.
\end{proof}

\begin{remark}\
\begin{enumerate}
\item In the theory of stochastic integration in M-type $2$ Banach spaces, a deterministic process $g$ is stochastically integrable if $g\in L^2(\R_+;\g(H,E))$. It would however be rather restrictive to assume this in Theorem \ref{thm:Doob}, unless $E$ has cotype $2$. Indeed, if $E$ does not have cotype $2$, we can find a function $g\in \g(L^2(\R_+;H), E)$ which is not in $L^2(\R_+,\g(H,E))$ (see \cite{NeeCMA,RS}). Clearly, $(M_t)_{t\geq 0}$ defined by $M_t = \int_0^t g(s) \, d W_H(s)$ is a continuous martingale which satisfies the conditions of Theorem \ref{thm:Doob}. In this sense, when $E$ is a \umd type $2$ Banach space, Theorem \ref{thm:Doob} covers more cases provided stochastic integration in \umd Banach spaces is considered.
\item If $(\F_t)_{t\geq 1}$ is a filtration which is generated by a cylindrical Brownian motion $(W_H(t))_{t\geq 0}$, then a stronger form of a representation theorem can be obtained (see \cite[Theorems 3.5 and 5.13]{NVW1}).
\item Similar techniques as above have been used in \cite{KunzeMart} in the setting of martingale solutions to stochastic evolution equations.
\end{enumerate}
\end{remark}

\appendix

\section{Stochastic integration\label{sec:stochint}}
Let $(\O,\A,\P)$ be a probability space with a complete filtration $\F=(\F_t)_{t\geq 0}$. Let $H$ be a real separable Hilbert space and let $E$ be a real Banach space.

In this section we recall some of the definitions and results of \cite{NVW1}. The space $\g(L^2(0,T;H), E)$ will be the space of {\em $\g$-radonifying operators} from $L^2(0,T;H)$ into $E$. For details we refer to \cite{NeeCMA,NVW1}.
The notions below also apply to the case where $H = \R$ and then we identify $\calL(H,E)$ with $E$.

A process $\Phi:[0,T]\times\O\to \calL(H,E)$ is said to {\em represent} $R\in L^p(\O;\g(L^2(0,T;H), E))$ if $\Phi$ is $H$-strongly measurable, weakly $L^p(\O;L^2(0,T;H))$ and for all $f\in L^2(0,T;H)$ and $x^*\in E^*$,
\[
\lb R(\omega)f, x^*\rb = \int_0^T \lb f(t), \Phi^*(t,\omega)x^*\rb_H\,dt \ \
\hbox{for almost all $\omega\in\O$}.
\]
In the above setting, the almost sure set in $\O$ can be chosen uniformly in $x^*\in E^*$ and $f\in L^2(0,T;H)$ (see \cite[Lemma 2.7]{NVW1}.
We will identity $R$ and $\Phi$ and write $\Phi\in L^p(\O;\g(L^2(0,T;H),E))$ instead of $R\in L^p(\O;\g(L^2(0,T;H),E))$ and the same for its norm. Sometimes to emphasize that $R$ is represented by $\Phi$, we will also write $R_{\Phi}$ instead of $R$.

If $\Phi$ is a linear combination of processes of the form $\one_{(a,b]} h\otimes  \xi$ with $0\leq a<b\leq T$ (with $b<\infty$), $h\in H$ and $\xi:\O\to E$ strongly $\F_a$-measurable, then $\Phi$ will be called an adapted step function. Clearly, $R_{\Phi}\in L^0(\O;\g(L^2(0,T;H), E))$. For every $p\in [0,\infty)$, the closure of the space of elements of the form $R_{\Phi}\in L^p(\O;\g(L^2(0,T;H), E))$, where $\Phi$ is an adapted step process will be denoted by $L^p_{\F}(\O;\g(L^2(0,T;H), E))$. If $\Phi$ is $H$-strongly measurable and adapted and represents $R\in L^p(\O;\g(L^2(0,T;H), E))$, then one has $R\in L^p_{\F}(\O;\g(L^2(0,T;H), E))$ (see \cite[Propositions 2.11 and 2.12]{NVW1}).

Let $(W_H(t))_{t\in \R_+}$ be a cylindrical Brownian motion with respect to $\F$. For $\Phi = \one_{(a,b]} h\otimes  \xi$ as before, let
\begin{equation}\label{eq:defstep}
\int_0^t \Phi \, d W_H = \big(W_H(b\wedge t)h - W_H(a\wedge t)h\big) \xi,  \ \ \ t\in [0,T].
\end{equation}
and extend this by linearity.

\begin{definition}\label{def:stochint}
Let $\Phi:[0,T]\times\Omega\to \calL(H,E)$ be an $H$-strongly measurable and adapted process which is scalarly in $L^0(\O;L^2(0,T;H))$. The process $\Phi$ is called stochastically integrable if
there exists a sequence $(\Phi_n)_{n\ge 1}$ of adapted step processes
such that:
\begin{enumerate}[(i)]
\item for all $x^*\in E^*$ we have $\limn  \Phi_n^* x^* = \Phi^*x^*$ in $L^0(\O;L^2(0,T;H))$;
\item there exists a process $\zeta\in L^0(\O;C_b([0,T];E))$ such that
$$ \zeta = \limn \int_0^\cdot \Phi_n(t)\,dW_H(t) \quad\hbox{in $L^0(\O;C_b([0,T];E))$};$$
\end{enumerate}
\end{definition}
By \cite[Proposition 17.6]{Kal} (ii) always implies that $(\Phi_n^* x^*)_{n\geq 1}$ is a Cauchy sequence in the space $L^0(\O;L^2(0,T;H))$. Moreover, by an additional approximation argument one can take each $\Phi_n$ to be a linear combination of  $\one_{(a,b]} h\otimes  \xi$, where $\xi$ is simple. Therefore, the above definition is equivalent to the assertion \cite[Theorem 5.9 (1)]{NVW1}.
An approximation in a finite dimensional subspace of $E$ shows that one may allow $a$ and $b$ to be stopping times in \eqref{eq:defstep}.

Let $\Phi:[0,T]\times\Omega\to \calL(H,E)$ be $H$-strongly measurable and adapted and such that $\Phi\one_{[0,T]}$ is scalarly in $L^0(\O;L^2(0,T;H))$ for every $0\leq T<\infty$. In this situation, $\Phi$ is called {\em locally stochastically integrable} if for every $0\leq T<\infty$, $\Phi\one_{[0,T]}$ is stochastically integrable. Clearly, in this case one can also find a pathwise continuous process $\zeta:\O\times\R_+\to E$ such that $\zeta(t) = \int_0^\infty \Phi(s)\one_{[0,t]}(s) \, d W_H(s)$.

Recall the following result \cite[Theorems 3.6, 5.9 and 5.12]{NVW1} (the case $p\in (0,1]$ can be either deduced from Lenglart's inequality or can be found in \cite{CoxVer2}). An overview on the theory of \umd spaces can be found \cite{Bu3}. As in \cite{Bu3} let $\beta_{p,E}$ with $p\in (1, \infty)$ be the \umd constant of $E$.
\begin{proposition}\label{prop:twosided}
Let $E$ be a \umd space. Let $\Phi:[0,T]\times\O\to \calL(H,E)$ be an $H$-strongly measurable and adapted process which is weakly in $L^0(\O;L^2(0,T))$. Then $\Phi$ is stochastically integrable if and only if $\Phi$ represents an element $R_{\Phi}\in L^0_{\F}(\O;\g(L^2(0,T;H), E))$. In this case, for all $p\in (0,\infty)$ one has
\[c_p^{-p} \beta_{p,E}^{-p} \E \|\Phi\|_{\g(L^2(0,T;H), E)}^p \leq \E \sup_{t\in[0,T]}\Big\|\int_0^t \Phi \, d W_H\Big\|^p \leq C_p^{p} \beta_{p,E}^{p} \E\|\Phi\|_{\g(L^2(0,T;H), E)}^p, \]
whenever one of the two terms is finite,
where $c_p, C_p$ are constants only depending on $p$
\end{proposition}
This result gives a useful description of the stochastically integrable processes. Moreover, as the $L^p$-estimate for the stochastic integral is a two-sided estimate, it is the right one. Estimation in the $\g$-norm is not always easy, but by now there a many techniques for that. If $E$ is a Hilbert space, then $\g(L^2(0,T;H), E)=L^2(0,T;\g(H,E))$. If $E = L^q$ a more explicit description of the $\g$-norm can be given as well (see \cite[Theorem 6.2]{NVWco}). Finally, we note that the Proposition \ref{prop:twosided} is naturally limited to \umd spaces, because the validity of the two-sided estimate for all adapted step processes already implies the \umd property (see \cite{Ga1}).

The next result will also be used and can be found in \cite[Lemma 5.4 and Theorem 5.5]{NVW1}.
\begin{proposition}\label{prop:L0cont}
Let $E$ be a \umd space and let $p\in(1, \infty)$. Let $\Phi:[0,T]\times\O\to \calL(H,E)$ be an $H$-strongly measurable and adapted process which is weakly in $L^0(\O;L^2(0,T))$. If $\Phi$ represents an element $R_{\Phi}\in L^0_{\F}(\O;\g(L^2(0,T;H), E))$,
then for all $\delta>0$ and $\varepsilon>0$ we have
\[
\P\Big(\sup_{t\in[0,T]}\Big\|\int_0^t \Phi \, d W_H\Big\|>\varepsilon\Big) \leq
\frac{C_p^p \beta_{p,E}^p\delta^p}{\varepsilon^p} + \P\big(\|\Phi\|_{\g(L^2(0,T;H), E)}\geq
\delta\big)
\]
and
\[
\P\big(\|\Phi\|_{\g(L^2(0,T;H), E)}
>\varepsilon\big) \leq \frac{c_p^p \beta_{p,E}^p \delta^p}{\varepsilon^p} + \P\Big(\sup_{t\in[0,T]}\Big\|\int_0^t \Phi \, d W_H\Big\|\geq
\delta\Big),
\]
where $c_p$ and $C_p$ are constants only depending on $p$.
\end{proposition}
This results should be interpreted as an equivalence of the convergence of a sequence of stochastic integrals $\Big(\int_0^\cdot \Phi_n\, d W_H\Big)_{n\geq 1}$ in $L^0(\O;C_b([0,T];E))$ and the convergence of $(\Phi_n)_{n\geq 1}$ in $L^0_{\F}(\O;\g(L^2(0,T;H), E)$. The constant $\beta_{p,E}$ plays an important role in the proof of Theorem \ref{thm:bddpaths} and is not stated explicitly in \cite{NVW1}, but can easily be deduced from its proof. Moreover, one can also check that the inequalities are valid with $c_2 = 1$ and $C_2=2$.

\section{Doob's regularization theorem}

The following is a simple extension of the Doob regularization theorem for martingales to the vector-valued situation, and it is used two times in the paper.

\begin{proposition}\label{prop:Doobreg}
Let $E$ be a Banach space and let $(\O, \A, \P)$ be a probability space with a complete filtration $\F = (\F_t)_{t\in \R_+}$. If $M:\R_+\times\O\to E$ is an $(\mathscr F_{t-})$-martingale, then $M$ has an $(\mathscr F_{t-})$-adapted modification where all paths are left-continuous and admit right limits.
\end{proposition}

\begin{proof}
It suffices to consider the case that $M$ is a martingale which is constant
after some time $T$. We can find a sequence of $\mathcal{F}_T$-simple functions $(f_n)_{n\geq 1}$
such that $M_T = \limn f_n$ in $L^1(\O;E)$. Let $(M^{n})_{n\geq 1}$ be the
sequence of martingales defined by $M^{n}_t = \E (f_n|\mathcal{F}_{t-})$. It
follows from the real-valued case of Doob's regularization theorem (see \cite[Theorem II.2.5 and Proposition II.2.7]{RY}) that each $M^{n}$ has an $(\mathscr F_{t-})$-adapted caglad modification
$\tilde{M}^{n}$. Furthermore, by \cite[Proposition 7.15]{Kal} and the
contractivity of the conditional expectation, for arbitrary $\delta>0$, we
have
\[\delta\,\P\Big(\sup_{t\in [0,T]}\|\tilde{M}^{n}_t-\tilde{M}^{m}_t\|>\delta\Big)\leq \E \|M^{n}_T - M^{m}_T\| \to 0 \ \ \text{if $n,m\to \infty$}.\]
Hence $(\tilde{M}^{n})_{n\geq 1}$ is a Cauchy sequence in probability in the
space $(CL([0,T];E), \|\cdot\|_{\infty})$ of caglad functions on $[0,T]$. Since $CL([0,T];E)$ is complete, $\tilde{M}^{n}$ is convergent to
some $\tilde{M}$ in $CL([0,T];E)$ in probability. For all $t\in [0,T]$ and all
$\delta>0$, we have
\[\begin{aligned}
\delta \P(\|\tilde{M}_t - M_t\|>\delta) & \leq \delta \P\Big(\|\tilde{M}_t -
\tilde{M}^n_t\|>\frac{\delta}{2}\Big) + \delta \P\Big(\|M^n_t - M_t
\|>\frac{\delta}{2}\Big)
\\ & \leq \delta \P\Big(\|\tilde{M}_t -
\tilde{M}^n_t\|>\frac{\delta}{2}\Big) + 2\E\|M^n_t - M_t \|
\\ & \leq \delta \P\Big(\|\tilde{M}_t -
\tilde{M}^n_t\|>\frac{\delta}{2}\Big) + 2\E\|f_n - M_T \|.
\end{aligned}\]
Since the latter converges to $0$ as $n$ tends to infinity and $\delta>0$ was
arbitrary, it follows that $\tilde{M}_t = M_t$ almost surely. This proves that
$\tilde{M}$ is the required modification of $M$.
\end{proof}

{\em Acknowledgment} -- The authors would like to thank Jan van Neerven helpful comments and careful reading. The authors thank the
anonymous referee for his/her help to improve the clarity of the exposition.

\def\polhk#1{\setbox0=\hbox{#1}{\ooalign{\hidewidth
  \lower1.5ex\hbox{`}\hidewidth\crcr\unhbox0}}}
  \def\polhk#1{\setbox0=\hbox{#1}{\ooalign{\hidewidth
  \lower1.5ex\hbox{`}\hidewidth\crcr\unhbox0}}} \def\cprime{$'$}
\providecommand{\bysame}{\leavevmode\hbox to3em{\hrulefill}\thinspace}

\end{document}